\documentclass[11pt, leqno]{amsart}
\usepackage{amsmath, amssymb, latexsym}
\usepackage{mathrsfs}
\usepackage{enumerate}
\usepackage{color}
\usepackage[colorlinks=true, pdfstartview=FitV, linkcolor=blue,%
citecolor=blue, urlcolor=blue]{hyperref}
\usepackage[all, knot]{xy}
\xyoption{arc}

\numberwithin{equation}{section}

\newtheorem{theorem}{Theorem}[section]
\newtheorem{corollary}[theorem]{Corollary}
\newtheorem{lemma}[theorem]{Lemma}
\newtheorem{proposition}[theorem]{Proposition}

\theoremstyle{definition}
\newtheorem{definition}[theorem]{Definition}

\newcommand{\C}{\mathbf{C}}

\newcommand{\Z}{\mathbf{Z}}
\newcommand{\h}{\mathfrak{h}}
\newcommand{\g}{\mathfrak{g}}

\title[Ringel-Hall algebra construction of quantum Borcherds-Bozec algebras]
{Ringel-Hall Algebra Construction of \\ Quantum Borcherds-Bozec
Algebras}

\author[Seok-Jin Kang]
{Seok-Jin Kang}

\address{Joeun Mathematical Research Institute,
441 Yeoksam-ro, Gangnam-gu, Seoul 06196, Korea}

         \email{soccerkang@hotmail.com}


\begin{document}

\begin{abstract}
We give the Ringel-Hall algebra construction of the positive half of
quantum Borcherds-Bozec algebras as the generic composition algebras
of quivers with loops.
\end{abstract}

\maketitle

\section*{Introduction}

\vskip 2mm

The {\it Hall algebra}, introduced by Steinitz \cite{St1901} and
rediscovered by Hall \cite{Hall1959}, is an associative algebra over
$\C$ with a basis consisting of isomorphism classes of finite
abelian $p$-groups. The finite abelian $p$-groups are parametrized
by partitions and the structure coefficients of the Hall algebra are
given by certain polynomials in $p$ with integral coefficients,
which are called the {\it Hall polynomials}. It turned out that
there is a close connection between the Hall algebras and the theory
of symmetric functions.

\vskip 3mm

In \cite{Ringel90}, Ringel generalized the notion of Hall algebras
to abelian categories with some finiteness conditions such as the
category of representations of a quiver. The  {\it Ringel-Hall
algebra} is an associative algebra over $\C$ with a basis consisting
of isomorphism classes of objects in a given abelian category, where
the multiplication is defined in terms of the space of extensions.
When we deal with the categories of representations of quivers
without loops, the Ringel-Hall algebras provide a realization of the
positive half of quantum groups associated with symmetric
generalized Cartan matrices \cite{Ringel90, Green95}. The
Ringel-Hall algebra construction of quantum groups is one of the
main inspirations for the Kashiwara-Lusztig crystal/canonical basis
theory \cite{Kas91, Lus90, Lus91}.

\vskip 3mm

Let us consider the quivers with loops. Then one can associate
symmetric Borcherds-Cartan matrices, which yield {\it Borcherds
algebras} or {\it generalized Kac-Moody algebras}. The Borcherds
algebras were introduced by Borcherds in his study of the Monstrous
Moonshine \cite{Bor88}. A special example of these algebras, the
Monster Lie algebra, played an important role in the proof of the
Moonshine Conjecture \cite{Bor92}. The quantum deformations of
Borcherds algebras and their modules were constructed in
\cite{Kang95}. In \cite{KS06}, the Ringel-Hall algebra construction
and the Kashiwara-Lusztig crystal/canonical basis theory were
generalized to the case of quantum Borcherds algebras (see also
\cite{JKK05}).

\vskip 3mm

The {\it Borcherds-Bozec algebras} are further generalizations of
Borcherds algebras. They are also defined by the generators and
relations coming from Borcherds-Cartan matrices, but they have far
more generators than Borcherds algebras. That is, for each simple
root, there are infinitely many generators whose degrees are
positive integral multiples of the given simple root. Thus in
addition to the Serre-type relations, we need to have the
Drinfel'd-type relations. The quantum Boecherds-Bozec algebras arise
as a natural algebraic structure behind the theory of perverse
sheaves on the representation varieties of quivers with loops and a
lot of interesting progresses are still under way (\cite{Bozec2014b,
Bozec2014c, BSV2016}, etc.).

\vskip 3mm

In this paper, we give the Ringel-Hall algebra construction of the
positive half of quantum Borcherds-Bozec algebras as the generic
composition algebras of quivers with loops. The main ingredients of
our work are Green's Theorem on symmetric bilinear forms of
Green-Lusztig algebras (Theorem \ref{thm:Green95}) and the
representations of quivers given in \eqref{eq:E_{il}} that
correspond to the higher degree generators of quantum
Borcherds-Bozec algebras.

\vskip 10mm

\section{Green-Lusztig algebras}

\vskip 2mm

Let $\mathscr {A}$ be an integral domain containing $\Z$ and an
invertible element $v$. Let $X$ be a set of alphabets (possibly
countably infinite) and let $\Lambda = \bigoplus_{x \in X} \Z
\alpha_{x}$ be the free abelian group on $X$ endowed with a
symmetric bilinear form $( \ , \ ) : \Lambda \times \Lambda
\rightarrow \Z$. The quadruple $(X,  ( \ , \ ), \mathscr{A}, v)$ is
called a {\it Green-Lusztig datum}. We write $\Lambda^{+} = \sum_{x
\in X} \Z_{\ge 0} \alpha_{x}$.

\vskip 3mm

\begin{definition} \ {\rm Let $(X, ( \ , \ ), {\mathscr A}, v)$ be a
Green-Lusztig datum. We say that an associative ${\mathscr
A}$-algebra $L$ is a {\it Green-Lusztig algebra belonging to the
class ${\mathscr L}(X, ( \ , \ ), {\mathscr A}, v)$} if the
following conditions are satisfied.
\begin{itemize}
\item[(a)] $L=\bigoplus_{\alpha \in \Lambda^{+}} L_{\alpha}$ is a
$\Lambda^{+}$-graded algebra such that

\begin{itemize}
\item[(i)] $L$ is generated by the elements $u_{x}$ $(x \in X)$,

\item[(ii)] $L_{0}={\mathscr A} \mathbf{1}$, where $\mathbf{1}$ is
the identity element of $L$.
\end{itemize}

\item[(b)] There is an ${\mathscr A}$-bilinear map $\delta: L
\rightarrow L \otimes_{\mathscr A} L$ such that

\begin{itemize}
\item[(i)]$ \delta(u_{x}) = u_{x} \otimes 1 + 1 \otimes u_{x} \ \
\text{for all} \ x \in X$,

\item[(ii)] $\delta$ is an ${\mathscr A}$-algebra homomorphism,
where the multiplication on $L \otimes_{\mathscr A} L$ is given by
$$(x_1 \otimes x_2) (y_1 \otimes y_2) := v^{(\beta_2, \gamma_1)} (x_1
y_1 \otimes x_2 y_2) \ \ \text{for} \ x_{i} \in L_{\beta_{i}}, \
y_{i} \in L_{\gamma_{i}} \ \ (i=1,2).$$
\end{itemize}

\item[(c)] There is a symmetric ${\mathscr A}$-bilinear form
$( \ , \ )_{L} : L \times L \rightarrow {\mathscr A}$ such that

\begin{itemize}
\item[(i)] $(L_{\alpha}, L_{\beta})_{L} =0$ if $\alpha \neq \beta$,

\item[(ii)] $(\mathbf{1}, \mathbf{1})_{L}=1$,

\item[(iii)] $(u_{x}, u_{x})_{L} \neq 0$ for all $x \in X$,

\item[(iv)] $(a, bc)_{L} = (\delta(a), b \otimes c)_{L}$ for all
$a,b,c \in L$, where $$(x_1 \otimes x_2, y_1 \otimes y_2)_{L} :=
(x_1, y_1)_{L} \, (x_2, y_2)_{L} \ \ (x_1, x_2, y_1, y_2  \in L).$$
\end{itemize}
\end{itemize}
}
\end{definition}

\vskip 3mm

Let $\beta = \sum_{x \in X} d_{x} \alpha_{x} \in \Lambda^{+}$ with
$ht(\beta):=\sum_{x \in X} d_{x} =r$. Set $$X(\beta):=\{w =(x_1,
\ldots, x_r) \mid \alpha_{x_1} + \cdots + \alpha_{x_r} = \beta \}.$$
If $L = \bigoplus_{\beta \in \Lambda^{+}} L_{\beta}$ is a
Green-Lusztig algebra in ${\mathscr L}(X, ( \ , \ ), {\mathscr A},
v)$, then $L_{\beta}$ is the ${\mathscr A}$-span of monomials of the
form $u_{w}= u_{x_1} \cdots u_{x_r}$ such that $w=(x_{1}, \ldots,
x_{r}) \in X(\beta)$. Note that if $w \in X(\beta)$, $w' \in
X(\beta')$ with $\beta \neq \beta'$, by (c), we have $(u_{w},
u_{w'})_{L} =0$.

\vskip 3mm

\begin{theorem} \cite{Green95} \label{thm:Green95} \ {\rm Let $\beta =
\sum_{x \in X} d_{x} \alpha_{x} \in \Lambda^{+}$ and $w, w' \in
X(\beta)$. Then there exists a Laurent polynomial $P_{w,w'}(t) \in
\Z[t, t^{-1}]$ such that for any Green-Lusztig datum $(X, ( \ , \ ),
\mathscr{A}, v)$ and any Green-Lusztig algebra in ${\mathscr L}(X, (
\ , \ ), {\mathscr A}, v)$, we have
$$(u_{w}, u_{w'})_{L} = P_{w, w'}(v) B_{\beta}(L),$$ where
$B_{\beta}(L) = \prod_{x \in X} (u_x, u_x)_{L}^{d_x}$. }
\end{theorem}

\vskip 2mm

\noindent {\it Remark.} The point is that $B_{\beta}(L)$ depends
only on $\beta$ and $L$.

\vskip 3mm

\begin{lemma} \cite{Green95} \label{lem:Green95} \ {\rm Let $L$ be
a Green-Lusztig algebra in ${\mathscr L}(X, ( \ , \ ), {\mathscr A},
v)$ and let $u=\sum_{w \in X(\beta)} c_{w} u_{w} \in L$ $(c_{w} \in
{\mathscr A})$. Then $u \in \text{rad}( \ , \ )_{L}$ if and only if
$$\sum_{w \in X(\beta)} c_{w} P_{w,w'}(v)=0 \ \ \text{for all} \ w' \in
X(\beta), \ \beta \in \Lambda^{+}.$$ }
\end{lemma}

\vskip 8mm

\section{Quantum Borcherds-Bozec algebras}

\vskip 2mm

Let $I$ be an index set (possibly countably infinite). A square
matrix $A=(a_{ij})_{i,j \in I}$ is called an {\it even symmetrizable
Borcherds-Cartan matrix} if

\begin{itemize}
\item[(i)] $a_{ii}=2, 0, -2, -4, \ldots$,

\item[(ii)] $a_{ij} \in \Z_{\le 0}$ for $i \neq j$,

\item[(iii)] $a_{ij}=0$ if and only if $a_{ji}=0$,

\item[(iv)] there is a diagonal matrix $D=\text{diag} (s_i \in
\Z_{>0} \mid i \in I)$ such that $DA$ is symmetric.
\end{itemize}

\vskip 3mm

Set $I^{\text{re}}:=\{ i \in I \mid a_{ii}=2 \}$, the set of {\it
real indices} and $I^{\text{im}}:= \{ i \in I \mid a_{ii} \le 0 \}$,
the set of {\it imaginary indices}. We denote by $I^{\text{iso}}: =
\{ i \in I \mid a_{ii}=0 \}$ the set of {\it isotropic indices}.

\vskip 3mm

A {\it Borcherds-Cartan datum} consists of

\begin{itemize}
\item[(a)] an even symmetrizable Borcherds-Cartan matrix
$A=(a_{ij})_{i,j \in I}$,

\item[(b)] a free abelian group $P$, the {\it weight lattice},

\item[(c)] $P^{\vee}:=Hom(P, \Z)$, the {\it dual weight lattice},

\item[(d)] $\Pi = \{ \alpha_i \in P \mid i \in I \}$, the set of
{\it simple roots},

\item[(e)] $\Pi^{\vee} = \{ h_i \in P^{\vee} \mid i \in I \}$, the
set of {\it simple coroots}
\end{itemize}
satisfying the following conditions
\begin{itemize}
\item[(i)] $\langle h_i, \alpha_j \rangle = a_{ij}$ for $i, j \in
I$,

\item[(ii)] $\Pi$ is linearly independent over $\C$,

\item[(iii)] for every $i \in I$, there is an element $\varpi_{i} \in
P$ such that $\langle h_j, \varpi_i \rangle = \delta_{ij}$ for all
$j \in I$.
\end{itemize}

\vskip 3mm We denote by $R:=\bigoplus_{i \in I} \Z \alpha_i$ the
{\it root lattice} and set $R^{+} := \sum_{i \in I} \Z_{\ge 0}
\alpha_i$.

\vskip 3mm

Let $\h: = \C \otimes_{\Z} P^{\vee}$. Since $A$ is symmetrizable and
$\Pi$ is linearly independent, there is a non-degenerate symmetric
bilinear form $( \ , \ )$ on $\h^*$ such that
\begin{equation} \label{eq:h-bilinear}
(\alpha_i, \lambda) = s_i \, \langle h_i, \lambda \rangle \ \
\text{for all} \ i \in I, \, \lambda \in \h^*.
\end{equation}

\vskip 3mm

Let $v$ be an indeterminate and set
$$v_i = v^{s_i}, \ \ v_{(i)} = v^{(\alpha_i, \alpha_i) / 2}, \ \
[n]_{i} = \dfrac{v_{i}^{n} - v_{i}^{-n}}{v_{i} - v_{i}^{-1}}.$$ Note
that $v_i = v_{(i)}$ if $i \in I^{\text{re}}$.

\vskip 3mm

Let $I^{\infty}:=(I^{\text{re}} \times \{1\}) \cup (I^{\text{im}}
\times \Z_{>0})$. We will often identify $I^{\text{re}} \times
\{1\}$ with $I^{\text{re}}$. Let $\Lambda := \bigoplus_{(i,l) \in
I^{\infty}} \Z \alpha_{il}$ be the free abelian group on
$I^{\infty}$. Then we have a symmetric bilinear form $( \ , \ ):
\Lambda \times \Lambda \rightarrow \Z$ given by
\begin{equation} \label{eq:lambda-bilinear}
(\alpha_{ik}, \alpha_{jl}):= kl(\alpha_i, \alpha_j) \ \ \text{for
all} \ \ (i,k), (j,l) \in I^{\infty}.
\end{equation}
Then $(I^{\infty}, ( \ , \ ), \C(v), v)$ is a Green-Lusztig datum.

\vskip 3mm

Let ${\mathscr E}$ be the free associative algebra over $\C(v)$
generated by the symbols $e_{il}$ for $(i,l) \in I^{\infty}$. Set
$\text{deg} e_{il}:= l \alpha_i$ for $(i,l) \in I^{\infty}$. Then
${\mathscr E}$ becomes an $R^{+}$-graded algebra ${\mathscr E} =
\bigoplus_{\beta \in R^{+}} {\mathscr E}_{\beta}$, where ${\mathscr
E}_{\beta}$ is the $\C(v)$-span of monomials of the form $e_{i_1,
l_1} \cdots e_{i_r, l_r}$ such that $l_1 \alpha_{i_1} + \cdots + l_r
\alpha_{i_r} = \beta$. We will denote by $|u|$ the degree of a
homogeneous element $u$ in ${\mathscr E}$.

\vskip 3mm

Define a {\it twisted multiplication} on ${\mathscr E} \otimes
{\mathscr E}$ by
\begin{equation} \label{eq:twisted}
(x_1 \otimes x_2) \, (y_1 \otimes y_2) = v^{(|x_2|, |y_1|)} x_1 y_1
\otimes x_2 y_2
\end{equation}
and a {\it co-multiplication} $\delta: {\mathscr E} \rightarrow
{\mathscr E} \otimes {\mathscr E}$ by
\begin{equation}\label{eq:co-mult}
\delta(e_{il}) = \sum_{m+n=l} v_{(i)}^{mn} e_{im} \otimes e_{in} \ \
\text{for all} \ \ (i,l) \in I^{\infty}.
\end{equation}
Since ${\mathscr E}$ is the free associative algebra on $\{e_{il}
\mid (i,l) \in I^{\infty} \}$, the map $\delta$ can be extended to a
well-defined algebra homomorphism.

\vskip 3mm

\begin{proposition} \cite{Bozec2014b, Bozec2014c, Lus00, Ringel96}
\label{prop:bozec2014b} \ {\rm For any family $\nu =
(\nu_{il})_{(i,l) \in I^{\infty}}$ of non-zero elements in $\C(v)$,
there exists a bilinear form $( \ , \ )_{L}: {\mathscr E} \times
{\mathscr E} \rightarrow \C(v)$ such that
\begin{itemize}
\item[(a)] $(x, y)_{L} =0$ if $|x| \neq |y|$,

\item[(b)] $(\mathbf{1}, \mathbf{1})_{L} = 1$,

\item[(c)] $(e_{il}, e_{il})_{L} = \nu_{il}$ for all $(i,l) \in
I^{\infty}$,

\item[(d)] $(x, yz)_{L} = (\delta(x), y \otimes z)$ for all $x,y,z
\in {\mathscr E}$.

\end{itemize}
}
\end{proposition}

\vskip 3mm

We define $\widehat {U}$ to be the associative algebra over $\C(v)$
generated by the elements $K_{i}^{\pm 1}$ $(i \in I)$, $e_{il},
f_{il}$ $((i,l) \in I^{\infty})$ with defining relations
\begin{equation} \label{eq:rels}
\begin{aligned}
& K_{i} K_{i}^{-1} = K_{i}^{-1} K_{i} =\mathbf{1}, \ \ K_{i} K_{j} =
K_{j} K_{i} \qquad (i, j \in I), \\
& K_{i} e_{jl} K_{i}^{-1} = v_{i}^{l a_{ij}} e_{jl}, \ \ K_{i}
f_{jl} K_{i}^{-1} = v_{i}^{-l a_{ij}} f_{jl}  \qquad (i \in I, (j,l)
\in
I^{\infty}), \\
& \sum_{k=0}^{1-l a_{ij}} (-1)^{k} e_{i}^{(k)} e_{jl}\, e_{i}^{(1- l
a_{ij} - k)} =0 \quad  \text{for} \ \ i \in I^{\text{re}}, \ i \neq
(j,l),\\
& \sum_{k=0}^{1- l a_{ij}} (-1)^{k} f_{i}^{(k)}  f_{jl}\, f_{i}^{(1
- l a_{ij}-k)} =0 \quad  \text{for} \ \ i \in I^{\text{re}}, \ i
\neq
(j,l), \\
& [e_{ik}, e_{jl}] =0  \ \ \text{if} \ a_{ij}=0.
\end{aligned}
\end{equation}

\noindent Here, we use the notation $e_{i}^{(k)} = e_{i}^k /
[k]_{i}!$, $f_{i}^{(k)} = f_{i}^{k} / [k]_{i}!$ for $i \in
I^{\text{re}}$.

\vskip 3mm

The algebra $\widehat{U}$ is endowed with the co-multiplication
$\Delta: \widehat{U} \rightarrow \widehat{U} \otimes \widehat{U}$
given by
\begin{equation} \label{eq:comult}
\begin{aligned}
& \Delta(K_{i}) = K_{i} \otimes K_{i}, \\
& \Delta(e_{il}) = \sum_{m+n=l} v_{(i)}^{mn}\, e_{im}\, K_{i}^n
\otimes
e_{in}, \\
& \Delta(f_{il}) = \sum_{m+n=l} v_{(i)}^{-mn} f_{im} \otimes
K_{i}^{-m} f_{in}.
\end{aligned}
\end{equation}

\vskip 2mm \noindent
 We will use Sweedler's notation to write
 $$\Delta(x)=\sum x_{(1)} \otimes x_{(2)} \ \ \text{for} \ \ x \in
 \widehat{U}.$$

 \vskip 3mm

 Let ${\widehat U}^{+}$ be the subalgebra of $\widehat{U}$ generated
 by $e_{il}$'s $((i,l) \in I^{\infty})$.

 \vskip 3mm

 \begin{proposition} \cite{Bozec2014b, Lus00, Ringel97} \
 {\rm

(a) If $i \in I^{\text{re}}$, $i \neq (j,l)$, then the elements
$$\sum_{k=0}^{1- l a_{ij}} (-1)^k e_{i}^{(k)} e_{jl} e_{i}^{(1- l
a_{ij} -l)}$$ lie in the radical of $( \ , \ )_{L}$.

\vskip 2mm

(b) If $a_{ij}=0$, then the elements $[e_{ik}, e_{jl}]$ $(k, l \ge
1)$ lie in the radical of $( \ , \ )_{L}$. }
\end{proposition}

\vskip 2mm \noindent Hence the bilinear form $( \ , \ )_{L}$ is
well-defined on $\widehat{U}^{+}$.

\vskip 3mm

Let ${\widehat U}^{\ge 0}$ be the subalgebra of $\widehat{U}$
generated by $\widehat{U}^{+}$ and $K_{i}^{\pm 1}$ $(i \in I)$. We
extend the bilinear form $( \ , \ )_{L}$ to $\widehat{U}^{\ge 0}$
via
\begin{equation} \label{eq:bilinearU+}
(x K_{i}, y K_{j})_{L} = v_{i}^{a_{ij}} (x, y)_{L}= v_{j}^{a_{ji}}
(x, y)_{L} \ \ \text{for all} \ x, y \in \widehat{U}^{+}, \ i, j \in
I.
\end{equation}

\vskip 3mm

Let $\omega: \widehat{U} \rightarrow \widehat{U}$ be the involution
defined by
\begin{equation} \label{eq:involution}
e_{il} \mapsto f_{il}, \quad f_{il} \mapsto e_{il}, \quad K_{i}
\mapsto K_{i}^{-1}.
\end{equation}

\noindent Then the subalgebra $\widehat{U}^{-}$ generated by
$f_{il}$'s $((i,l) \in I^{\infty})$ is endowed with a symmetric
bilinear form $( \ , \ )_{L}$ by setting
\begin{equation} \label{eq:bilinearU-}
(x, y)_{L} =(\omega(x), \omega(y))_{L} \ \ \text{for all} \ x, y \in
\widehat{U}^{-}.
\end{equation}

\vskip 3mm

Following the Drinfel'd double process, we take the algebra
$\widetilde{U}$ to be the quotient of $\widehat{U}$ by the relations
\begin{equation} \label{eq:Drinfeld}
\sum (a_{(1)}, b_{(2)})_{L} \, \omega(b_{(1)}) \, a_{(2)} = \sum
(a_{(2)}, b_{(1)})_{L} \, a_{(1)}\, \omega(b_{(2)}) \ \ \text{for
all} \ \  a, b \in \widehat{U}^{\ge 0}.
\end{equation}

\vskip 3mm

\begin{definition} \label{def:qBozec}
The {\it quantum Borcherds-Bozec algebra} $U_{v}(\g)$ associated
with the Borcherds-Cartan datum $(A, P, P^{\vee}, \Pi, \Pi^{\vee})$
is the quotient algebra of $\widetilde{U}$ by the radical of $(\ , \
)_{L}$ restricted to $\widetilde{U}^{-} \times \widetilde{U}^{+}$.
\end{definition}

\vskip 3mm

Thus we have $U^{\pm}_{v}(\g) = \widetilde{U}^{\pm} \big /
\text{rad} ( \ , \ )_{L}$, where $U^{+}_{v}(\g)$ (resp.
$U^{-}_{v}(\g)$) is the subalgebra of $U_{v}(\g)$ generated by
$e_{il}$'s (resp. $f_{il}$'s) for $(i,l) \in I^{\infty}$.

\vskip 3mm

From now on, we assume that
\begin{equation} \label{eq:assumption}
(e_{il}, e_{il})_{L} \in 1 + v^{-1} \Z_{\ge 0} [[v^{-1}]] \ \
\text{for all} \ i \in I^{\text{im}} \setminus I^{\text{iso}}, \, l
\ge 1.
\end{equation}
Then $( \ , \ )_{L}$ is non-degenerate on ${\mathscr
E}(i):=\bigoplus_{l \ge 1} {\mathscr E}_{l \alpha_i}$.

\vskip 3mm

\begin{proposition} \cite{Bozec2014b, Bozec2014c}
\label{prop:Bozec-a} \ {\rm For each $i \in I^{\text{im}}$ and $l
\ge 1$, there exists a unique element $s_{il} \in {\mathscr E}_{l
\alpha_i}$ such that

\begin{itemize}

\item[(i)]  $\langle s_{i,1}, \ldots, s_{i,l} \rangle = \langle
e_{i,1}, \ldots, e_{i,l} \rangle$ as algebras,

\vskip 2mm

\item[(ii)] $(s_{il}, z)_{L} =0$ for all $z \in \langle e_{i,1},
\ldots, e_{i, l-1} \rangle$,

\vskip 2mm

\item[(iii)]  $s_{il} - e_{il} \in \langle e_{i,1}, \ldots, e_{i, l-1}
\rangle$,

\vskip 2mm

\item[(iv)]  $\delta(s_{il}) = s_{il} \otimes 1 + 1 \otimes s_{il}$,

\vskip 2mm

\item[(v)]  $\Delta(s_{il}) = s_{il} \otimes 1 + K_{i}^l \otimes
s_{il}$.
\end{itemize}
}
\end{proposition}

\vskip 3mm

\begin{proposition} \cite{Bozec2014b, Bozec2014c} \label{prop:Bozec-b} \
{\rm $U^{\pm}_{v}(\g) = \widetilde{U}^{\pm}$. In particular, $( \ ,
\ )_{L}$ is non-degenerate on $U^{\pm}_{v}(\g)$.}
\end{proposition}

\vskip 3mm

Combining Proposition \ref{prop:Bozec-a} and Proposition
\ref{prop:Bozec-b}, we obtain

\vskip 2mm

\begin{corollary} \label{cor:U+} \
{\rm  The algebra $U^{+}_{v}(\g)$ is a non-degenerate Green-Lusztig
algebra belonging to the class ${\mathscr L}(I^{\infty}, ( \ , \
)_{L}, \C(v), v)$. }
\end{corollary}

\begin{proof} \ Note that $U^{+}_{v}(\g)$ is generated by $s_{il}$
and that $\delta(s_{il})= s_{il} \otimes 1 + 1 \otimes s_{il}$ for
$(i,l) \in I^{\infty}$. Since $( \ , \ )_{L}$ is non-degenerate on
${\mathscr E}(i)$ for each $i \in I^{\text{im}}$, we have
$$(s_{il}, s_{il})_{L} = (s_{il}, e_{il})_{L} \neq 0,$$ which proves
our claim.
\end{proof}

\vskip 2mm

\noindent {\it Remark}. The algebra ${\mathscr E}$ is also a
(degenerate) Green-Lusztig algebra belonging to the class ${\mathscr
L}(I^{\infty}, ( \ , \ )_{L}, \C(v), v)$.

\vskip 8mm

\section{Ringel-Hall algebras}

\vskip 2mm

Let $I$ be an index set (possibly countably infinite) and let
$R=\bigoplus_{i \in I} \Z \alpha_i$ be the free abelian group on
$I$. Let $Q=(I, \Omega)$ be a quiver, where $I$ is the set of
vertices and $\Omega$ is the set of arrows. We have the functions
$\text{out}, \text{in}: \Omega \rightarrow I$ defined by
$$\text{out}(h) \overset{h} \longrightarrow \text{in}(h) \ \ \text{for}
\ \ h \in \Omega.$$

\vskip 2mm

\begin{definition} \label{def:repn} \ {\rm
Let $\mathbf{k}$ be a field and let $Q=(I, \Omega)$ be a quiver. A
{\it representation of $Q$ over $\mathbf{k}$} consists of

\begin{itemize}
\item[(i)] a family of finite dimensional $\mathbf{k}$-vector spaces
$M=(M_{i})_{i \in I}$ such that $M_{i}=0$ for all but finitely many
$i$,

\item[(ii)] a family of $\mathbf{k}$-linear maps $x=(x_{h}:
M_{\text{out}(h)} \rightarrow M_{\text{in}(h)})_{h \in \Omega}$.
\end{itemize}
}
\end{definition}

\vskip 2mm

For simplicity, we often write $(M,x)$ for a representation of $Q$.

\vskip 3mm

\begin{definition} \label{def:morphism} \
{\rm Let $(M,x)$ and $(N,y)$ be representations of a quiver $Q=(I,
\Omega)$. A {\it morphism} $\phi:(M,x) \rightarrow (N,y)$ is a
family of $\mathbf{k}$-linear maps $\phi=(\phi_i: M_i \rightarrow
N_i)_{i \in I}$ such that, for all $h \in \Omega$, the following
diagram is commutative.
\begin{equation} \label{diagram:morphism}
\xymatrix{ M_{\text{out}(h)} \ar[d]^-{x_{h}}
\ar[r]^-{\phi_{\text{out(h)}}}&
N_{\text{out}(h)} \ar[d]^-{y_{h}} \\
M_{\text{in}(h)}  \ar[r]^-{\phi_{\text{in}(h)}} & N_{\text{in}(h)}}
\end{equation}

}
\end{definition}

\vskip 3mm

Let $M=(M_{i})_{i \in I}$ be a representation of $Q$. We define the
{\it dimension vector} of $M$ by
\begin{equation}
\underline{\text{dim}}\, M = \sum_{i \in I} (\text{dim}_{\mathbf{k}}
M_{i})\, \alpha_i \in R^{+}.
\end{equation}

\vskip 2mm

Let $M$ and $N$ be representations of $Q$. The (non-symmetric) {\it
Euler form} of $M$ and $N$ is defined by
\begin{equation} \label{eq:euler-a}
\langle M, N \rangle = \dim_{\mathbf{k}} Hom_{\mathbf{k} Q} (M,N) -
\dim_{\mathbf{k}} Ext_{\mathbf{k} Q}^{1} (M,N).
\end{equation}

\vskip 2mm

On the other hand, for $\alpha = \sum_i d_i \alpha_i, \,
\beta=\sum_i d_i' \alpha_i \in R^{+}$, we define
\begin{equation} \label{eq:euler-b}
\langle \alpha, \beta \rangle = \sum_i (1 - g_i) d_i d_i' -
\sum_{\substack {i \neq j \\ i \rightarrow j}} c_{ij} d_{i} d_{j}',
\end{equation}
where $g_i$ is the number of loops at $i$ and $c_{ij}$ denotes the
number of arrows from $i$ to $j$ in $\Omega$.

\vskip 3mm

The following lemma is well-known (see, for example, \cite{CB, Hu}).

\vskip 3mm

\begin{lemma} {\rm \ Let $M$ and $N$ be representations of $Q$. Then we have
$$\langle M,
N \rangle = \langle \underline{\dim} \, M, \underline{\dim} \, N
\rangle.$$
}
\end{lemma}

\vskip 3mm

For $\alpha, \beta \in R^{+}$, we define
\begin{equation} \label{eq:Euler-b}
(\alpha, \beta) := \langle \alpha, \beta \rangle + \langle \beta,
\alpha \rangle.
\end{equation}
In particular, we have
\begin{equation}\label{eq:Cartan}
(\alpha_i, \alpha_j) = \begin{cases} 2(1-g_i) \ \ & \text{if} \ \
i=j, \\
-c_{ij} -c_{ji} \ \ & \text{if} \ \ i \neq j.
\end{cases}
\end{equation}

\vskip 2mm

Hence we obtain a symmetric Borcherds-Cartan matrix $A_{Q}=
(a_{ij})_{i,j \in I} = ((\alpha_i, \alpha_j))_{i,j \in I}$ with $R$
as the root lattice. We will denote by $U_{v}(\g_{Q})$ the quantum
Borcherds-Bozec algebra associated with $A_{Q}$.

\vskip 3mm

Let $\mathbf{k}$ be a finite field with $q$ elements and choose a
complex number $v=v_{\mathbf{k}} \in \C$ such that $v^2 = q$. Then
$(I, (\ , \ ), \C, v)$ is a Green-Lusztig datum.

\vskip 3mm

\begin{definition} \ {\rm
The {\it Ringel-Hall algebra} $H_{\mathbf{k}}(Q)$ is the associative
algebra over $\C$ with a basis consisting of isomorphism classes of
representations of $Q$ endowed with the multiplication defined by
\begin{equation} \label{eq:mult}
[M] \, [N] := \sum_{L} v^{\langle \dim M,\, \dim N \rangle}
\alpha_{M,N}^{L} \, [L],
\end{equation}
where $[M]$ denotes the isomorphism class of $M$ and
\begin{equation} \label{eq:coeff}
\alpha_{M,N}^{L} = \# \{X \subset L \mid X \cong N, \, L/X \cong M
\}.
\end{equation}
}
\end{definition}

\vskip 3mm

Let $\alpha \in R^{+}$ and let $H_{\mathbf{k}}(Q)_{\alpha}$ be the
$\C$-span of the isomorphism classes with $\underline{\dim} \, M =
\alpha$. Then $H_{\mathbf{k}}(Q) = \bigoplus_{\alpha \in R^{+}}
H_{\mathbf{k}}(Q)_{\alpha}$ becomes an $R^{+}$-graded algebra
(\cite{Green95, Ringel96}, etc).

\vskip 3mm

We define a {\it twisted} algebra structure on $H_{\mathbf{k}}(Q)
\otimes H_{\mathbf{k}}(Q)$ by
\begin{equation} \label{eq:twisted-Hall}
([M_1] \otimes [M_2]) \, ([N_1] \otimes [N_2]) =
v^{(\underline{\dim}\, M_2,\, \underline{\dim}\, N_1)} ([M_1]\,
[N_1] \otimes [M_2]\, [N_2])
\end{equation}
and a $\C$-linear map $\delta: H_{\mathbf{k}}(Q) \rightarrow
H_{\mathbf{k}}(Q) \otimes H_{\mathbf{k}}(Q)$ by
\begin{equation} \label{eq:comult-Hall}
\delta([L]) = \sum_{M,N} v^{\langle \underline{\dim}\, M, \,
\underline{\dim} \, N \rangle} \alpha_{M,N}^{L} \, \dfrac{a_{M}
a_{N}}{a_{L}} ([M] \otimes [N]),
\end{equation}
where $a_{M} = \#(Aut_{\mathbf{k}Q} (M))$.

\vskip 3mm

\begin{proposition} \cite{Green95} \label{prop:Green95} \hfill

\vskip 2mm

{\rm (a) $\delta: H_{\mathbf{k}}(Q) \rightarrow H_{\mathbf{k}}(Q)
\otimes H_{\mathbf{k}}(Q)$ is a $\C$-algebra homomorphism.

\vskip 2mm

(b) There exists a non-degenerate symmetric bilinear form $( \ , \
)_{G}: H_{\mathbf{k}}(Q) \times H_{\mathbf{k}}(Q) \rightarrow \C$
defined by
$$([M], [N])_{G} = \delta_{[M], [N]} \dfrac{1}{a_{M}}.$$

(c) We have
$$(x, yz)_{G} = (\delta(x), y \otimes z)_{G} \ \ \text{for all} \ \
x,y,z \in H_{\mathbf{k}}(Q).$$ }
\end{proposition}

\vskip 3mm

Let $(i,l) \in I^{\infty}$. If $i \in I^{\text{re}}$, we define
$E_{i}$ to be the unique simple representation of $Q$ with dimension
vector $\alpha_i$. By \eqref{eq:comult-Hall}, we see that
$$\delta([E_i]) = [E_i] \otimes 1 + 1 \otimes [E_i].$$

\vskip 3mm

Assume that $i \in I^{\text{im}}$. For each $l \ge 1$, we define a
representation $(E_{i,l}, x)$ of $Q$ by setting
\begin{equation} \label{eq:E_{il}}
\begin{aligned}
 (E_{i,l})_{j} &= \begin{cases} {\mathbf k}^{l} \ \ & \text{if} \ \
j=i, \\
0 \ \ & \text{if} \ \ j \neq i,
\end{cases} \\
x_{h} &= 0 \ \ \text{for all} \ h \in \Omega.
\end{aligned}
\end{equation}

\vskip 3mm

Note that
\begin{equation} \label{eq:aut}
\begin{aligned}
& a_{E_{i,l}} = \#(GL_{l}(\mathbf{k})) = (q^l -1) (q^{l} -q)
\cdots (q^{l}-q^{l-1}) = v^{\frac{3}{2} l (l-1)} (v^2 -1)^l \, [l]!, \\
& \alpha_{E_{i,m},\, E_{i,n}}^{E_{i,m+n}} = \#Gr_{\mathbf{k}}
\binom{m+n}{m} = v^{mn} {{m+n} \brack {m}}.
\end{aligned}
\end{equation}
Therefore we obtain
$$
\begin{aligned}
\delta([E_{i,l}]) &=\sum_{m+n=l} v^{\langle \underline{\dim}\,
E_{i,m},\, \underline{\dim}\, E_{i,n} \rangle} \alpha_{E_{i,m},\,
E_{i,n}}^{E_{i,m+n}} \, \dfrac{a_{E_{i,m}}
a_{E_{i,n}}}{a_{E_{i,m+n}}} \,
[E_{i,m}] \otimes [E_{i,n}] \\
&= \sum_{m+n=l} v^{mn \langle \alpha_i, \alpha_i \rangle} v^{mn} \,
{{m+n} \brack {m}} \, \dfrac{v^{\frac{3}{2}(m(m-1)+n(n-1))}(v^2
-1)^{m+n}\, [m]! \, [n]!}{v^{\frac{3}{2}((m+n)(m+n-1))} (v^2
-1)^{m+n} \,
[m+n]!} \\
&= \sum_{m+n=l} v^{mn(1-g_i-2)} \, [E_{i,m}] \otimes [E_{i,n}]\\
&= \sum_{m+n=l} v^{mn(-1-g_i)} \, [E_{i,m}] \otimes [E_{i,n}].
\end{aligned}
$$

\vskip 3mm

Hence for all $(i,l) \in I^{\infty}$, we have
\begin{equation} \label{eq:comult-Comp}
\delta([E_{i,l}]) = \sum_{m+n=l} v^{mn(-1-g_i)} [E_{i,m}] \otimes
[E_{i,n}].
\end{equation}
Moreover, by \eqref{eq:aut}, we see that
\begin{equation} \label{eq:norm_a}
([E_{i,l}], [E_{i,l}])_{G} = \dfrac{1}{a_{E_{i,l}}} \in v^{-2 {l^2}}
(1 + v^{-1} \Z [[v^{-1}]]) \quad \text{for all} \ \ (i,l) \in
I^{\infty}.
\end{equation}

\vskip 3mm

Set ${\mathbf{e}}_{il} := v^{l^2} [E_{i,l}] \in H_{\mathbf{k}}(Q)$.
Then by \eqref{eq:comult-Comp}, we obtain
\begin{equation} \label{eq:comult-comp}
\delta({\mathbf e}_{il}) = \sum_{m+n=l} v_{(i)}^{mn} \, {\mathbf
e}_{im} \otimes {\mathbf e}_{in},
\end{equation}
where $v_{(i)} = v^{\langle \alpha_i, \alpha_i \rangle} =
v^{1-g_i}$. Moreover, it is easy to see that
\begin{equation}
{\mathbf e}_{il}, {\mathbf e}_{il})_{G} \in 1 + v^{-1} \Z[[v^{-1}]]
\quad \text{for all} \ \ (i,l) \in I^{\infty}.
\end{equation}

\vskip 2mm

\begin{definition} \label{def:comp} \ {\rm  The subalgebra
$C_{\mathbf{k}}(Q)$ of $H_{\mathbf{k}}(Q)$ generated by ${\mathbf
e}_{i,l}$ $((i,l)\in I^{\infty})$ is called the {\it composition
algebra of $Q$ over $\mathbf{k}$}. }
\end{definition}

\vskip 2mm

By \eqref{eq:comult-comp}, we see that $\delta(C_{\mathbf{k}}(Q))
\subset C_{\mathbf{k}}(Q) \otimes_{\mathbf{k}} C_{\mathbf{k}}(Q)$
and hence $C_{\mathbf{k}}(Q)$ is a bi-algebra.

\vskip 3mm

For each $i \in I$, set $H_{\mathbf{k}}(i):=\bigoplus_{l \ge 1}
H_{\mathbf{k}}(Q)_{l \alpha_i}$. Then the restriction of $( \ , \
)_{G}$ to $H_{\mathbf{k}}(i)$ is non-degenerate. Hence, as in
\cite[Proposition 2.16]{Bozec2014b}, we have:

\vskip 3mm

\begin{proposition} \label{prop:s_{il}} \
{\rm For each $(i,l) \in I^{\infty}$, there exists a unique element
${\mathbf s}_{il} \in H_{\mathbf{k}}(Q)$ such that
\begin{itemize}
\item[(a)] $\langle {\mathbf s}_{i,1}, \ldots, {\mathbf s}_{i,l} \rangle = \langle
{\mathbf e}_{i,1}, \ldots , {\mathbf e}_{i,l} \rangle$ as algebras,

\vskip 2mm

\item[(b)] $({\mathbf s}_{il},\,  x)_{G}=0$ for all $x \in \langle {\mathbf e}_{i,1}, \ldots, {\mathbf e}_{i, l-1}
\rangle$,

\vskip 2mm

\item[(c)] ${\mathbf s}_{il} - {\mathbf e}_{il} \in \langle {\mathbf e}_{i,1}, \ldots,
{\mathbf e}_{i,l-1} \rangle$,

\vskip 2mm

\item[(d)] $\delta({\mathbf s}_{il}) = {\mathbf s}_{il} \otimes 1 + 1 \otimes
{\mathbf s}_{il}$.
\end{itemize}
}
\end{proposition}

\vskip 3mm

As in the proof of Corollary \ref{cor:U+}, by (b) and (c), we see
that
$$({\mathbf s}_{il}, {\mathbf s}_{il})_{G} \neq 0 \ \ \text{for all} \ (i,l) \in
I^{\infty}.$$

\vskip 2mm

Therefore we obtain:

\begin{proposition} \label{prop:comp-a} \
{\rm The composition algebra $C_{\mathbf{k}}(Q)$ is a Green-Lusztig
algebra belonging to the class $\mathscr{L}(I^{\infty}, ( \ , \
)_{G}, \C, v_{\mathbf{k}})$. }
\end{proposition}

\vskip 3mm

The following proposition and its corollary show that the quantum
Serre relations hold in the composition algebra $C_{\mathbf{k}}(Q)$.

\vskip 3mm

\begin{proposition} \label{prop:Serre-comp} \
{\rm For every finite field $\mathbf{k}$, the following relations
hold.

\vskip 2mm

(a) If $a_{ij}=0$, then
$$[E_{i,k}][E_{j,l}] = [E_{j,l}][E_{i,k}].$$

\vskip 2mm

(b) If $i \in I^{\text{re}}$ and $i \neq (j,l)$, then we have
$$\sum_{k=0}^{1-l a_{ij}} (-1)^k [E_{i}]^{(k)} [E_{j,l}] [E_{i}]^{(1
- l a_{ij} -k)} =0,$$ where $[E_{i}]^{(k)} :=[E_{i}]^{k} \big /
[k]!$. }
\end{proposition}

\begin{proof} \ Set $v=v_{\mathbf{k}}$. If $a_{ii}=0$, by the
duality, we have
$$[E_{i,k}][E_{i,l}] = \sum_{L} \alpha_{E_{i,k}, \, E_{i,l}}^{L} \, [L]
= \sum_{L} \alpha_{E_{i,l}^*, \, E_{i,k}^*}^{L*} \, [L^*] = \sum_{L}
\alpha_{E_{i,l}, \, E_{ik}}^{L} \, [L] = [E_{i,l}] [E_{i,k}].$$

\vskip 2mm

If $i \neq j$, $a_{ij}=0$ implies $c_{ij}=c_{ji}=0$. Thus
$Hom_{\mathbf{k} Q} (E_{i,k}, E_{j,l})=0$ and
$$[E_{i,k}] [E_{j,l}] = [E_{i,k} \oplus E_{j,l}] =
[E_{j,l}][E_{i,k}],$$ which proves (a).

\vskip 3mm To prove (b), by induction, we first verify
$$[E_{i}]^{(k)}= \dfrac{1}{[k]!}[E_{i}]^{k} = v^{k(k-1)} [E_{i}^{\oplus k}].$$

\vskip 2mm

Now we have
\begin{equation*}
\begin{aligned}
& [E_{i}]^{(k)}  [E_{j,l}]  = v^{k(k-1)} v^{\langle k \alpha_i, l
\alpha_j \rangle} \sum_{L} \alpha_{E_{i}^{\oplus k},\,
E_{j,l}}^{L}\,
[L] \\
& = v^{k(k-1)-kl c_{ij}} \sum_{L} \alpha_{E_{i}^{\oplus k},\,
E_{j,l}}^{L} \, [L],
\end{aligned}
\end{equation*}
where $L$ runs over $\mathbf{k} Q$-modules containing a submodule
$X$ such that $$X \cong E_{j,l}, \quad L \big / X \cong
E_{i}^{\oplus k}.$$ Since $Hom_{\mathbf{k}Q} (E_{i}, E_{j,l})=0$,
such a submodule $X$ is unique and hence
$$\alpha_{E_{i}^{\oplus k},\, E_{j,l}}^{L} =1 \ \ \text{for all} \
L.$$

\vskip 2mm \noindent It follows that
$$[E_{i}]^{(k)} [E_{j,l}] = v^{k(k-1) - kl c_{ij}} \sum_{L} \, [L],$$
where $L$ contains a (unique) submodule $X$ such that $X \cong
E_{j,l}$, $L \big / X \cong E_{i}^{\oplus k}$.

\vskip 3mm

Hence for any $n \ge 0$, we have
\begin{equation*}
\begin{aligned}
& [E_{i}]^{(k)}  [E_{j,l}] [E_{i}]^{(n)} = (v^{k(k-1) - kl c_{ij}}
\sum_{L} \, [L]) \, E_{i}^{(n)} \\
&= v^{k(k-1) - kl c_{ij}} \, v^{\langle k \alpha_i + l \alpha_j, n
\alpha_i \rangle } \, v^{n(n-1)} \, \sum_{L} \sum_{P} \alpha_{L,
E_{i}^{\oplus n}}^{P} \, [P] \\
& = v^{k(k-1)+n(n+1)+kn - kl c_{ij} - ln c_{ji}} \, \sum_{P}
(\sum_{L} \alpha_{L, E_{i}^{\oplus n}}^{P}) \, [P],
\end{aligned}
\end{equation*}
where
$$\alpha_{L, E_{i}^{\oplus n}}^{P} =\# \{Y \subset P \mid
Y \cong E_{i}^{\oplus n}, \, P \big / Y \cong L \}.$$

\vskip 2mm

Set \begin{equation*}
\begin{aligned}
& K_{P}:= \bigcap_{h: i \rightarrow j} \text{Ker} (x_{h}: {\mathbf
k}^{\oplus (k+n)} \rightarrow {\mathbf k}^{l}) \subset P_{i}, \\
& J_{P}:= \sum_{h':j \rightarrow i} \text{Im}(x_{h'}: {\mathbf
k}^{l} \rightarrow {\mathbf k}^{\oplus (k+n)}) \subset P_{i}, \\
& m_{P}:=\dim K_{P}, \quad n_{P}:= \dim J_{P}.
\end{aligned}
\end{equation*}

\vskip 2mm \noindent Then $P \big/ Y \cong L$ if and only if

\begin{itemize}
\item[(i)] $\underline{\dim}\, Y = n \alpha_i$,

\vskip 2mm

\item[(ii)] $x_{h}=0$ for all $h:i \rightarrow j$,

\vskip 2mm

\item[(iii)] $\text{Im} x_{h'} \subset Y$ for all $h':j \rightarrow
i$.
\end{itemize}

\vskip 2mm

\noindent Hence we have
\begin{equation*}
\begin{aligned}
& \beta_{P,n}:= \sum_{L} \alpha_{L, E_{i}^{\oplus n}}^{P}
 = \sum_{L} \# \{Y \subset P \mid Y \cong E_{i}^{\oplus n}, \ P/Y
\cong L \} \\
& = \# \{\text{$n$-dimensional subspaces $Y$ of $K_{P}$ containing
$J_{P}$} \} \\
& = \# \{\text{$(n-n_{P})$-dimensional subspaces of $K_{P}/J_{P}$}
\}
\\
&=\# Gr_{\mathbf{k}}\binom{m_{P}-n_{P}}{n-n_{P}}
=v^{(m_{P}-n)(n-n_{P})} {{m_{P}-n_{P}} \brack {n-n_{P}}},
\end{aligned}
\end{equation*}
which implies
$$[E_{i}]^{(k)} [E_{j,l}][E_{i}]^{(n)} = v^{k(k-1) + n(n-1) + kn -
klc_{ij} - ln c_{ji}} \sum_{P} v^{(m_{P}-n)(n-n_{P})} {{m_{P}-n_{P}}
\brack {n-n_{P}}} \, [P].$$

\vskip 2mm

By setting $n=1 - la_{ij} -k$ and summing up, we obtain
$$\sum_{k=0}^{1- l a_{ij}} (-1)^{k} [E_{i}]^{(k)} [E_{j,l}]
[E_{i}]^{(1-la_{ij}-k)}=\sum_{P\,:\, J_{P} \subset K_{P}} \gamma_{P}
[P],$$ where
\begin{equation*}
\begin{aligned}
\gamma_{P} & = \sum_{k=0}^{1 - l a_{ij}} (-1)^k v^{k(k-1)+n(n-1)+kn
- kl c_{ij} -ln c_{ji} + (m_{P}-n)(n-n_{P})} {{m_{P}-n_{P}} \brack
{n-n_{P}}} \\
& = \sum_{n=0}^{1 - l a_{ij}} (-1)^{1-l a{ij}-n} v^{l
c_{ji}(1-la_{ij})+ n(-2l c_{ji} + m_{P}+n_{P} -1) -m_{P}
n_{P}}{{m_{P}-n_{P}} \brack
{n-n_{P}}} \\
& = (-1)^{1-la_{ij}} v^{l c_{ji}(1-l a_{ij}) - m_{P} n_{P}}
\sum_{n=n_{P}}^{m_{P}} (-1)^n v^{n(-2l c_{ji} + m_{P} + n_{P} -1)}
{{m_{P}-n_{P}} \brack {n-n_{P}}}.
\end{aligned}
\end{equation*}

Let $$\gamma_{P}^{0}:=\sum_{n=n_{P}}^{m_{P}} (-1)^n v^{n(-2l c_{ji}
+ m_{P} + n_{P} -1)} {{m_{P}-n_{P}} \brack {n-n_{P}}}.$$

\noindent Note that $\dim \text{Im}\, x_{h} \le l c_{ij}$ and $n_{P}
= \dim J_{P} \le l c_{ji}$. Hence we have
$$m_{P}=\dim K_{P} \ge 1 - l a_{ij} - l c_{ij} = 1 + l c_{ji} > n_{P}
$$
and obtain
\begin{equation*}
\begin{aligned}
& (m_{P} - n_{P} -1) -(-2l c_{ji} + m_{P} + n_{P} -1) =2(l c_{ji} -
n_{P}) \ge 0, \\
& (-2l c_{ji} + m_{P} + n_{P} -1) -(-m_{P} + n_{P} +1) =2(m_{P} - l
c_{ji} -1) \ge 0,
\end{aligned}
\end{equation*}
which yield
$$-m_{P} + n_{P} + 1 \le -2l c_{ji} + m_{P} + n_{P} -1 \le m_{P} -
n_{P} -1.$$

\vskip 2mm

It is well-known that
$$\sum_{k=0}^m (-1)^k v^{dk} {m \brack k} =0$$ for all $m\ge 1$,
$-m+1 \le d \le m-1$, $d \equiv m-1 \ (\text{mod}\,2)$ (see, for
example, \cite{Kas91}).

\vskip 3mm Therefore, since $-2l c_{ji} + m_{P} + n_{P} -1 \equiv
m_{P} - n_{P} -1 \ (\text{mod} \, 2)$, we have
\begin{equation*}
\begin{aligned}
\gamma_{P}^{0} & = \sum_{n=n_{P}}^{m_{P}} (-1)^n v^{n(-2l c_{ji} +
m_{P} + n_{P} -1)} {{m_{P}-n_{P}} \brack {n - n_{P}}} \\
& = \sum_{r=0}^{m_{P}-n_{P}} (-1)^{r + n_{P}} v^{r + n_{P}(-2l
c_{ji} + m_{P} + n_{P} -1)} {{m_{P}-n_{P}} \brack r}\\
& =(-1)^{n_{P}} v^{n_{P}(-2l c_{ji} + m_{P} + n_{P} -1)}
\sum_{r=0}^{m_{P}-n_{P}} (-1)^r v^{r(-2lc_{ji} + m_{P} + n_{P} -1)}
{{m_{P} - n_{P}} \brack r} \\
&=0.
\end{aligned}
\end{equation*}
Hence we conclude $\gamma_{P}=0$ for all $P$, which proves our
assertion.
\end{proof}

\vskip 3mm

\begin{corollary} \label{cor:Serre-comp} \
{\rm For every finite field $\mathbf{k}$, the following relations
hold.

\vskip 2mm

(a) If $a_{ij}=0$, then
$${\mathbf e}_{ik} \, {\mathbf e}_{jl} = {\mathbf e}_{jl} \, {\mathbf e}_{ik}.$$

\vskip 2mm

(b) If $i \in I^{\text{re}}$ and $i \neq (j,l)$, then we have
$$\sum_{k=0}^{1-l a_{ij}} (-1)^k {\mathbf e}_{i}^{(k)}
{\mathbf e}_{jl} \, {\mathbf e}{i}^{(1 - l a_{ij} -k)} =0,$$ where
${\mathbf e}_{i}^{(k)} :={\mathbf e}_{i}^{k} \big / [k]!$. }
\end{corollary}

\vskip 8mm

\section{Ringel-Hall algebra construction of $U_{v}^{+}(\g_{Q})$}

\vskip 2mm

Let $K$ be an infinite set of mutually non-isomorphic finite fields.
For each $\mathbf{k} \in K$, choose $v_{\mathbf{k}} \in \C$ such
that $v_{\mathbf{k}}^2 = \#(\mathbf{k})$ and set
\begin{equation} \label{eq:generic-Hall}
H(Q):= \prod_{\mathbf{k} \in K} H_{\mathbf{k}}(Q),
\end{equation}
the {\it generic Ringel-Hall algebra}.

\vskip 3mm

Let $v$ be an indeterminate. Then $H(Q)$ can be regarded as a $\C[v,
v^{-1}]$-module via
$$v^{\pm 1} \longmapsto (v_{\mathbf{k}}^{\pm 1})_{\mathbf{k} \in K}.$$
For each $(i,l) \in I^{\infty}$, let $E_{i,l;\mathbf{k}}$ be the
representation of $Q$ over $\mathbf{k}$ defined in \eqref{eq:E_{il}}
and let ${\mathbf s}_{i,l; \mathbf{k}}$ be the element in
$H_{\mathbf{k}}(Q)$ given in Proposition \ref{prop:s_{il}}. Set
\begin{equation}
{\mathbf E_{i,l}}:= ({\mathbf e}_{i,l; \mathbf{k}})_{\mathbf{k} \in
K}=(v_{\mathbf{k}}^{l^2}[E_{i,l; \mathbf{k}}])_{\mathbf{k} \in K},
\qquad {\mathbf S}_{i,l}:=({\mathbf s}_{i,l;\mathbf{k}})_{\mathbf{k}
\in K}.
\end{equation}

\vskip 3mm

\begin{definition} \label{def:generic-comp} \
{\rm The {\it generic composition algebra of $Q$} is the $\C[v,
v^{-1}]$-subalgebra $C(Q)$ of $H(Q)$ generated by ${\mathbf
E}_{i,l}$ for all $(i,l) \in I^{\infty}$.}
\end{definition}

\vskip 3mm

By Proposition \ref{prop:comp-a}, the generic composition algebra
$C(Q)$ is a Green-Lusztig algebra belonging to the class
$\mathscr{L}(I^{\infty}, ( \ , \ )_{G}, \C[v, v^{-1}], v)$. We now
state and prove the main theorem of this paper.

\vskip 3mm

\begin{theorem} \label{thm:main} \ {\rm
There exists a natural isomorphism of $\C(v)$-bialgebras $$\Phi:
U_{v}^{+}(\g_{Q}) \rightarrow \C(v) \otimes_{\C[v, v^{-1}]} C(Q)$$
given  by
$$e_{il} \longmapsto {\mathbf E}_{i,l} \ \ \text{for all}  \ (i,l) \in I^{\infty}.$$
}
\end{theorem}

\begin{proof} \ By Corollary \ref{cor:Serre-comp}, $\Phi$ defines
a surjective $\C(v)$-bialgebra homomorphism.

\vskip 3mm

To prove the injectivity of $\Phi$, we will use Theorem
\ref{thm:Green95}. Let $\Lambda: = \bigoplus_{(i,l) \in I^{\infty}}
\Z \alpha_{i,l}$ and let $\Lambda^{+}:= \sum_{(i,l) \in I^{\infty}}
\Z_{\ge 0} \alpha_{i,l}$. For $\beta = \sum d_{i,l} \, \alpha_{i,l}
\in \Lambda^{+}$, set
$$I^{\infty}(\beta):=\{ w = ((i_1, l_1), \ldots,
(i_r, l_r)) \mid \alpha_{i_1, l_1} + \cdots + \alpha_{i_r, l_r} =
\beta \}.$$

\vskip 2mm

For each $w=((i_1, l_1), \ldots, (i_r, l_r)) \in I^{\infty}(\beta)$,
we denote the generating monomials by
\begin{equation*}
\begin{aligned}
& s_{w} := s_{i_1, l_1} \cdots s_{i_r, l_r} \in
U_{v}^{+}(\g_{Q})_{\beta}, \\
& {\mathbf s}_{w; \mathbf{k}}:= {\mathbf s}_{i_1, l_1;\mathbf{k}}
\cdots {\mathbf s}_{i_r, l_r;\mathbf{k}} \in
C_{\mathbf{k}}(Q)_{\beta}, \\
& {\mathbf S}_{w}:=({\mathbf s}_{w; \mathbf{k}})_{\mathbf{k} \in
K}={\mathbf S}_{i_1, l_1} \cdots {\mathbf S}_{i_r, l_r}  \in
C(Q)_{\beta}.
\end{aligned}
\end{equation*}
Then one can see that $s_{w}$ is mapped onto ${\mathbf S}_{w}$ under
the homomorphism $\Phi$.

\vskip 3mm

By Theorem \ref{thm:Green95}, for all $\beta = \sum d_{i,l} \,
\alpha_{i,l} \in \Lambda^{+}$, $w, w' \in I^{\infty}(\beta)$ and
$\mathbf{k} \in K$, there exists a polynomial $P_{w, w'}(t) \in
\Z[t, t^{-1}]$ such that
\begin{itemize}
\item[(i)] $(s_{w}, s_{w'})_{L} = P_{w, w'}(v) \prod_{(i,l) \in
I^{\infty}} (s_{i,l}, s_{i,l})_{L}^{d_{i,l}}$,

\vskip 2mm

\item[(ii)] $({\mathbf s}_{w; \mathbf{k}}, {\mathbf s}_{w'; \mathbf{k}})_{G} =
P_{w,w'}(v_{\mathbf{k}}) \prod_{(i,l)\in I^{\infty}} ({\mathbf
s}_{i,l;\mathbf{k}}, {\mathbf s}_{i,l;\mathbf{k}})_{G}^{d_{i,l}}.$
\end{itemize}

\vskip 3mm

Let $u=\sum_{w} c_{w}(v) s_{w} \in \text{Ker}\, \Phi \subset
U_{v}^{+}(\g_{Q})$. Thus \, $\sum_{w} c_{w}(v)  \,{\mathbf
S}_{w}=0$, which implies
$$\sum_{w} c_w(v_{\mathbf{k}}) \, {\mathbf s}_{w;\mathbf{k}} =0 \ \ \text{for
all} \ \mathbf{k} \in K.$$ Then, for any $w' \in I^{\infty}(\beta)$,
we have
$$0=\sum_{w} c_{w}(v_{\mathbf{k}}) ({\mathbf s}_{w;\mathbf{k}},\,
{\mathbf s}_{w';\mathbf{k}})_{G} =\sum_{w} c_{w}(v_{\mathbf{k}})
P_{w, w'}(v_{\mathbf{k}}) \prod_{(i,l) \in I^{\infty}} ({\mathbf
s}_{i,l;\mathbf{k}}, {\mathbf s}_{i,l;\mathbf{k}})_{G}^{d_{i,l}}.$$
It follows that
$$\sum_{w} c_{w}(v_{\mathbf{k}}) P_{w,w'}(v_{\mathbf{k}})=0 \ \
\text{for all} \ w' \in I^{\infty}(\beta) \, \mathbf{k} \in K.$$
Therefore $\sum_{w} c_{w}(v) P_{w,w'}(v)=0$ for all $w' \in
I^{\infty}(\beta)$.

\vskip 3mm

By Lemma \ref{lem:Green95}, we have
$$u = \sum_{w} c_{w}(v) s_{w} \in \text{rad}( \ , \ )_{L}.$$
Since $( \ , \ )_{L}$ is non-degenerate on $U_{v}^{+}(\g_{Q})$, we
conclude $u=0$ and hence $\Phi$ is injective.
\end{proof}

\vskip 5mm

{\it Acknowledgements}. The author would like to thank Professor
Tristan Bozec for very helpful suggestions. He is also grateful to
Professor Jae-Hoon Kwon and Professor Young-Tak Oh for their help
and support.

\end{document}